\documentclass[11pt]{article}

\usepackage{amsmath,amsthm,amssymb}
\usepackage[usenames,dvipsnames]{xcolor}
\usepackage{enumerate}
\usepackage{graphicx}
\usepackage{cite}
\usepackage{comment}
\usepackage[margin=1in]{geometry}
\usepackage{bbm}

\usepackage[pdftitle={Closure of LQG},
  pdfauthor={Andres Contreras-Hip and Ewain Gwynne},
colorlinks=true,linkcolor=NavyBlue,urlcolor=RoyalBlue,citecolor=PineGreen,bookmarks=true,bookmarksopen=true,bookmarksopenlevel=2,unicode=true,linktocpage]{hyperref}

\theoremstyle{plain}
\newtheorem{thm}{Theorem}[section]
\newtheorem{theorem}[thm]{Theorem}
\newtheorem{cor}[thm]{Corollary}
\newtheorem{lemma}[thm]{Lemma}
\newtheorem{prop}[thm]{Proposition}

\def\@rst #1 #2other{#1}
\newcommand\MR[1]{\relax\ifhmode\unskip\spacefactor3000 \space\fi
  \MRhref{\expandafter\@rst #1 other}{#1}}
\newcommand{\MRhref}[2]{\href{http://www.ams.org/mathscinet-getitem?mr=#1}{MR#2}}

\theoremstyle{definition}
\newtheorem{defn}[thm]{Definition}
\newtheorem{remark}[thm]{Remark}

\numberwithin{equation}{section} 

\newcommand{\dsb}{\begin{adjustwidth}{2.5em}{0pt}
\begin{footnotesize}}
\newcommand{\dse}{\end{footnotesize}
\end{adjustwidth}}

\newcommand{\ssb}{\begin{adjustwidth}{2.5em}{0pt}}
\newcommand{\sse}{\end{adjustwidth}}

\newcommand{\aryb}{\begin{eqnarray*}}
\newcommand{\arye}{\end{eqnarray*}}
\def\alb#1\ale{\begin{align*}#1\end{align*}}
\def\allb#1\alle{\begin{align}#1\end{align}}
\newcommand{\eqb}{\begin{equation}}
\newcommand{\eqe}{\end{equation}}
\newcommand{\eqbn}{\begin{equation*}}
\newcommand{\eqen}{\end{equation*}}

\newcommand\e{\varepsilon}

\newcommand\R{\mathbb{R}}
\newcommand\Z{\mathbb{Z}}

\let\originalleft\left
\let\originalright\right
\renewcommand{\left}{\mathopen{}\mathclose\bgroup\originalleft}
\renewcommand{\right}{\aftergroup\egroup\originalright}

\title{Bounds on the distance exponent for higher-dimensional Liouville first passage percolation}
\author{Andres A. Contreras Hip\thanks{University of Chicago}  \qquad \qquad Zijie Zhuang\thanks{University of Pennsylvania}}

\begin{document}

\maketitle

\newcommand{\Cupper}{{\hyperref[eqn-bilip-def]{\mathfrak C_*}}}

\begin{abstract}
For $\xi \geq 0$ and $d \geq 3$, the higher-dimensional Liouville first passage percolation (LFPP) is a random metric on $\epsilon \mathbb{Z}^d$ obtained by reweighting each vertex by $e^{\xi h_\epsilon(x)}$, where $h_\epsilon(x)$ is a continuous mollification of the whole-space log-correlated Gaussian field. This metric generalizes the two-dimensional LFPP, which is related to Liouville quantum gravity. We derive several estimates for the set-to-set distance exponent of this metric, including upper and lower bounds and bounds on its derivative with respect to $\xi$. In the subcritical region for $\xi$, we derive estimates for the fractal dimension and show that it is continuous and strictly increasing with respect to $\xi$. In particular, our result is an important step towards proving a technical assumption made in \cite{cg-support}. These are also the first bounds on the distance exponent for LFPP in higher dimensions.
\end{abstract}

\setcounter{tocdepth} {1}
\tableofcontents

%\begin{figure}[ht!]
%\begin{center}
%\includegraphics[width=0.75\textwidth]{Pic1b.pdf} 
%\caption{\label{fig1} Type caption here
%}
%\end{center}
%\vspace{-3ex}
%\end{figure}

%\zjcomment{
%Main goal: Obtain upper and lower bounds for $\lambda(\xi)$ for all $\xi \geq 0$, where $\lambda(\xi)$ is the distance exponent for the exponential metric of the log-correlated Gaussian field defined as in~\cite{gp-lfpp}. \\ 
%State of the art in two dimensions: Bounds for $\lambda(\xi)$ are obtained in~\cite{dzz-heat-kernel, ghs-dist-exponent, dg-fractal-dim, ding-goswami-watabiki, gp-lfpp, ang-discrete-lfpp, lfpp-pos}. See~\cite{ddg-metric-survey} for a review. The best known bounds are obtained in~\cite{gp-lfpp, ang-discrete-lfpp}, with further improvements in~\cite{lfpp-pos} and for small $\gamma$ in~\cite{ding-goswami-watabiki}.\\
%In this paper, we will extend the results in these papers to $d \geq 3$. We will also show that the distance exponent $\lambda(\xi)$ is increasing with respect to $d$.}

\section{Introduction}

\textit{Liouville first passage percolation} (LFPP) with parameter $\xi \geq 0$ is a family of random metrics on $\mathbb{C}$ obtained by integrating $e^{\xi h_\epsilon}$ along paths, where $h_\epsilon$ is a continuous mollification of the two-dimensional \textit{Gaussian free field} (GFF). It was shown in~\cite{DDDF-tightness, GM-uniqueness, DG-supercritical-tightness, DG-supercritical-unique} that, under appropriate normalization and topology, these metrics converge to a limiting metric as $\epsilon \to 0$, known as the \textit{Liouville quantum gravity} (LQG) metric. We refer to~\cite{Sheffield-ICM, ddg-metric-survey} for reviews on this metric and its relation to random planar maps. One of the most important questions about the LQG metric is to calculate its fractal dimension in the subcritical region. In fact, the problem is equivalent to computing the LFPP distance exponent in the subcritical region (see \cite[Theorem 1.5]{dg-fractal-dim} and \cite[Corollary 1.7]{GP-KPZ-metric}). However, the only explicit value known for the LFPP distance exponent is when $\xi = 1/\sqrt{6}$, where it equals $1/6$. This result follows from Miller and Sheffield's construction of the $\sqrt{8/3}$-LQG metric~\cite{lqg-tbm1, lqg-tbm2, lqg-tbm3} and its connection to the Brownian map~\cite{legall-uniqueness, miermont-brownian-map}. We refer to Theorem~\ref{thm:bound-2d} and Figure~\ref{fig:bounds} (left) for the best known bounds on the LFPP distance exponent as obtained in~\cite{dg-fractal-dim, gp-lfpp, ang-discrete-lfpp}.

Recently, there has been growing interest in studying analogues of LQG in higher dimensions; see e.g.~\cite{Baptiste2022high, Sturm2024high, dgz-exponential-metric}. It is therefore natural to investigate the distance exponent and fractal dimension of the higher-dimensional LFPP, which is the focus of this paper. Let $h$ be the whole-space \textit{log-correlated Gaussian field} (LGF) which is a random generalized function on $\mathbb{R}^d$ satisfying
$$
{\rm Cov}[(h, f_1), (h, f_2)] = \int_{\mathbb{R}^d \times \mathbb{R}^d} f_1(x) f_2(y) \log \frac{1}{|x-y|} dx dy,
$$
where $f_1$ and $f_2$ are Schwartz functions with average 0. For $d = 2$, the LGF coincides with the GFF, whereas for $d \geq 3$, they are different. Throughout this paper, we normalize $h$ so that its average over the unit sphere is 0. We refer to~\cite{DRSV-lgf, fgf-survey} for the well-definedness and basic properties of the LGF. For $\e>0$ and $x \in \mathbb{R}^d$, we define $h_\e(x)$ as the continuous mollification of $h$, given by its average over the box $x + [-\e, \e]^d$; see Remark~\ref{rmk:mollification} for discussions on other mollifications. For a parameter $\xi \geq 0$, the higher-dimensional LFPP is a random metric on the rescaled lattice $\e \mathbb{Z}^d$ obtained by reweighting each vertex by $e^{\xi h_\e(x)}$. Specifically, for a discrete path $P: \{1,2,\ldots, N \} \to \e \mathbb{Z}^d$, the length of $P$ is defined as
\begin{equation}\label{eq:def-length}
    L_h^{\e, \xi} (P) := \sum_{i=1}^{N} \e e^{\xi h_\e(P(i))}.
\end{equation}
For $z, w \in \e \mathbb{Z}^d$, the higher-dimensional LFPP metric is defined as
\begin{equation}\label{eq:def-metric}
    D_h^{\e, \xi} (z,w) := \inf_{P:z \to w} L_h^{\e, \xi} (P),
\end{equation}
where the infimum is taken over all discrete paths $P$ connecting $z$ and $w$ on $\e \mathbb{Z}^d$. For a subset $A$ of $\e \mathbb{Z}^d$, we also define the LFPP metric restricted to this set, denoted by $D_h^{\e, \xi}(\cdot,\cdot; A)$, by requiring the path $P$ in~\eqref{eq:def-metric} to stay in $A$.

Let $\mathbbm{S} = [0,1]^d$ and $\mathbbm{S}^\e = \mathbbm{S} \cap \e \mathbb{Z}^d$. The set-to-set distance exponent characterizes the behavior of the exponential metric, which is defined as
\begin{equation}\label{eq:def-lambda}
\lambda = \lambda(d, \xi) := \sup \{\alpha \in \mathbb{R}: \liminf_{\e \to 0}\mathbb{P}[D_h^{\e, \xi}(\partial_L \mathbbm{S}^\e, \partial_R \mathbbm{S}^\e;  \mathbbm{S}^\e) \leq \e^\alpha] = 1 \},
\end{equation}
where $\partial_L \mathbbm{S}^\e = (\{0\} \times [0,1]^{d-1}) \cap \e \mathbb{Z}^d $ and $\partial_R \mathbbm{S}^\e = (\{1\} \times [0,1]^{d-1}) \cap \e \mathbb{Z}^d $ are the left and right boundaries of $\mathbbm{S}^\e$, respectively. In Section~\ref{subsec:bound}, we give bounds on $\lambda$ and its derivative with respect to $\xi$. In Section~\ref{subsec:fractal-dimension}, we define the fractal dimension in the subcritical region in terms of $\lambda$ and show that it is continuous and strictly increasing with respect to $\xi$.

\begin{remark}\label{rmk:mollification}
    \begin{itemize}
        \item (Different mollifications) In the above definition, we consider the continuous mollification $h_\epsilon$ by averaging $h$ over $[-\e, \e]^d$. In fact, it is easy to see that the set-to-set distance exponent is the same for many different mollifications. For instance, one can consider a mollification obtained by convolving $h$ with $\e^{-d} \rho(\cdot/\e)$, where $\rho$ is any probability measure satisfying $\iint_{\mathbb{R}^d \times \mathbb{R}^d} \log \frac{1}{|x-y|} \rho(dx) \rho(dy)<\infty$. One can also use the mollification defined by the white noise decomposition as considered in~\cite{dgz-exponential-metric, dgz-thick-point}. In fact, it follows from Lemma~\ref{lem:mollification} and \cite[Lemma 4.1]{dgz-thick-point} that all these mollifications of $h$ differ by at most $o(\log \e^{-1} )$ on $\mathbbm{S}^\e$ with probability $1 - o_\e(1)$ as $\e \to 0$. This in particular implies that when $d = 2$, our distance exponent agrees with the one considered in~\cite{dg-fractal-dim, gp-lfpp, ang-discrete-lfpp}. One can also define the LFPP metric by integrating along continuous paths as in~\cite{dgz-exponential-metric}, which does not affect the distance exponent; see \cite[Lemma 3.7]{dgz-exponential-metric}.

        \item (Convergence of the LFPP metric) It is expected that in the subcritical region of $\xi$, after appropriate normalization, the higher-dimensional LFPP metric converges under the uniform topology to a limiting metric that induces the Euclidean topology.  In~\cite{dgz-exponential-metric}, the authors proved the tightness of the higher-dimensional LFPP metric in the subcritical region for a different mollification defined by the white noise decomposition, although the uniqueness of the subsequential limit is still open. The behavior for large $\xi$ is more mysterious. There may be a phase similar to that of the 2D supercritical LQG metric~\cite{DG-supercritical-tightness, DG-supercritical-unique} where the limiting metric has an uncountable, dense set of singular points with zero Lebesgue measure. A recent work~\cite{dgz-thick-point}
        suggests that a new phase may appear for large $\xi$ in sufficiently high dimensions, where the set-to-set distance can be much smaller than the point-to-point distance between typical points. This phase does not occur in two dimensions.
        
    \end{itemize}
\end{remark}

\subsection{Upper and lower bounds on $\lambda$ and $\lambda'$}\label{subsec:bound}

Before stating the bounds on $\lambda$ for $d \geq 3$, let us first recall the best known bounds on $\lambda(2, \xi)$. Define 
$$
\underline{\rho}(\xi) = 
\begin{cases}
\max \left\{(\sqrt{\frac{5}{2}} - \frac{1}{\sqrt{6}}) \xi - \frac{\sqrt{15}-2}{6},0 \right\} \quad &\xi \leq \frac{1}{\sqrt 6},  \\
\max \left\{\frac{1}{4} - \frac{\xi^2}{2},0 \right\} \quad &\xi \geq \frac{1}{\sqrt 6},
\end{cases}
$$
and
$$
\overline{\rho}(\xi) = 
\begin{cases}
\min \left\{  \frac{1}{4} - \frac{\xi^2}{2} ,  \sqrt{2}\xi    \right\} \quad &\xi \leq \frac{1}{\sqrt 6},  \\
\min\left\{  \left(\sqrt{\frac{5}{2}} - \frac{1}{\sqrt 6} \right) \xi  - \frac{\sqrt{15} -2}{6}    ,   1     \right\} \quad &\xi \geq \frac{1}{\sqrt 6} .
\end{cases}
$$
The following bounds on $\lambda(2,\xi)$ are proved in \cite[Theorem 2.3]{gp-lfpp} and \cite[Theorem 1.8]{ang-discrete-lfpp}; see Figure~\ref{fig:bounds} (left).

\begin{theorem}[\cite{gp-lfpp, ang-discrete-lfpp}]\label{thm:bound-2d}
    For all $\xi \geq 0$, we have $\underline{\rho}(\xi) \leq \lambda(2, \xi) \leq \overline{\rho}(\xi)$.
\end{theorem}

It was also shown in~\cite{ding-goswami-watabiki} that $\lambda(2, \xi) \geq c \xi^{4/3}/\log(\xi^{-1})$ for some constant $c>0$. As mentioned before, one of the major open questions in LQG is to determine the value of $\lambda(2,\xi)$.

The following theorem is our main bound on $\lambda(d,\xi)$; see Figure~\ref{fig:bounds} (right) for the case $d = 3$.

\begin{theorem}\label{thm:bound}
    For all $d \geq 3$ and $\xi \geq 0$, we have $\underline{\rho}(\xi) \leq \lambda(d, \xi) \leq \xi \sqrt{2d-2}$.
\end{theorem}

\begin{figure}[ht!]
\begin{center}
\includegraphics[scale=0.7]{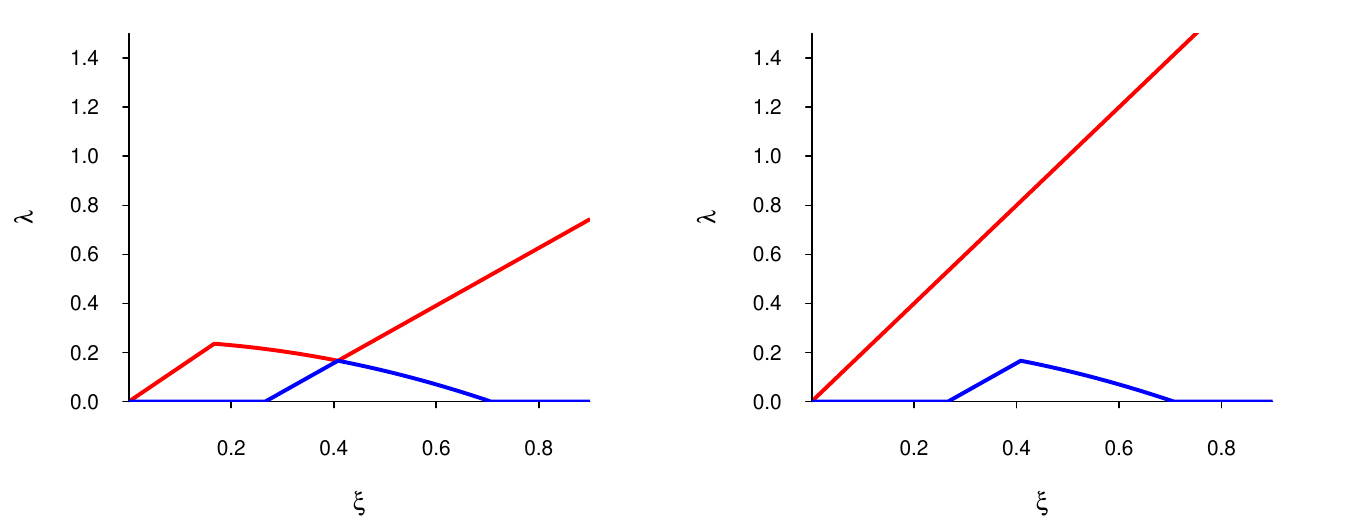}
\caption{\label{fig:bounds} \textbf{Left}: Graph of the upper and lower bounds on $\lambda(2,\xi)$ from Theorem~\ref{thm:bound-2d}. By the connection between $\sqrt{8/3}$-LQG and the Brownian map, we have $\lambda(2, 1/\sqrt{6}) = 1/6$. \textbf{Right}: Graph of the upper and lower bounds on $\lambda(3,\xi)$ from Theorem~\ref{thm:bound}. We do not know the explicit value of $\lambda(3, \xi)$ for any $\xi>0$.
}
\end{center}
\vspace{-3ex}
\end{figure}

The lower bound in Theorem~\ref{thm:bound} follows from the following lemma, which is based on the observation in \cite[Theorem 7.1]{fgf-survey} that restricting the $(d+1)$-dimensional LGF to $\mathbb{R}^d \times \{0\}$ yields a $d$-dimensional LGF. From this, we deduce that the $(d+1)$-dimensional LFPP distance should be upper-bounded by the $d$-dimensional LFPP distance and thus $\lambda(d,\xi)$ should be increasing with respect to $d$. Some caution is needed to handle the different ways of mollification in $d$- and $(d+1)$-dimensions; see Section~\ref{sec:restriction} for proof details.

\begin{lemma}\label{lem:dimen-restriction}
For all $\widetilde d > d \geq 2$ and $\xi \geq 0$, we have $\lambda(\widetilde d, \xi) \geq \lambda(d, \xi)$.
\end{lemma}

Now we prove Theorem~\ref{thm:bound} assuming Lemma~\ref{lem:dimen-restriction}.

\begin{proof}[Proof of Theorem~\ref{thm:bound} given Lemma~\ref{lem:dimen-restriction}]
    The lower bound in Theorem~\ref{thm:bound} follows from Theorem~\ref{thm:bound-2d} and Lemma~\ref{lem:dimen-restriction}. Next, we prove the upper bound in Theorem~\ref{thm:bound}. Note that $\mathbb{E} h_\e(x)^2 = - \log \e +O(1)$. Therefore, for any fixed $\delta>0$, with probability $1 - o_\e(1)$, there are at most $\e^{-1}$ number of vertices $v$ in $\mathbbm{S}^\e$ satisfying $h_\e(v) \leq (\sqrt{2d-2} +\delta) \log \e$. On this event, any path $P$ connecting $\partial_L \mathbbm{S}^\e$ and $\partial_R \mathbbm{S}^\e$ on $\mathbbm{S}^\e$ must contain at least $\e^{-1}$ number of vertices $v$ satisfying $h_\e(v) > (\sqrt{2d-2} + \delta) \log \e$. Therefore, $L^{\e,\xi}_h(P) \geq \e^{\xi (\sqrt{2d-2} + \delta)}$. This implies that $\lambda(d,\xi) \leq \xi (\sqrt{2d-2} +\delta)$. Taking $\delta \to 0$ yields the desired upper bound.\qedhere
    
\end{proof}

The following theorem gives a better lower bound for $\lambda$ when the dimension is sufficiently large, which follows as a consequence of~\cite{dgz-thick-point}. Note that when $d = 2$, it was shown in~\cite{lfpp-pos} that $\lambda(2, \xi) < 1$ for all $\xi \geq 0$. Theorem~\ref{thm:bound-high-dimen} shows that this inequality does not hold in sufficiently high dimensions. 

\begin{theorem}\label{thm:bound-high-dimen}
For any fixed $A>0$, there exists a constant $C>0$ depending on $A$ such that for all $d \geq C$ and $\xi \geq 0$, we have $\lambda(d, \xi) \geq A \xi - A^{-1}$.
\end{theorem}

\begin{proof}

By~\cite[Proposition 3.19]{dgz-thick-point}\footnote{The refined paths considered in Proposition 3.19 of~\cite{dgz-thick-point} at scale $n$ are defined on $\frac{1}{\lfloor \sqrt{d} \rfloor} 8^{-n} \mathbb{Z}^d$ and have length $O(1) \times 11^n$. The number 11 arises because at each step of refining, a vertex is expanded into a subpath of length 11. In fact, the same argument there can be applied to define refined paths on $\frac{1}{\lfloor \sqrt{d} \rfloor} 2^{-Mn} \mathbb{Z}^d$ that connect the left and right boundaries of $\mathbb{S}$ and each path has length at most $O(1) \times (2^M+3)^n$. By choosing $M$ large enough and taking $\epsilon = \frac{1}{\lfloor \sqrt{d} \rfloor} 2^{-Mn}$, we get a set of paths $\mathcal{P}$ on $\mathbb{S}^\epsilon$ with length at most $\epsilon^{-1-A^{-1}}$. For any fixed $M$ and sufficiently large $d$ (which may depend on $M$), using the same argument as in Proposition 3.19 of~\cite{dgz-thick-point}, we derive that with high probability, there exists a path $P$ in $\mathcal{P}$ such that $\widetilde h_{Mn}(x) \geq (A+1)Mn$ for all vertices on $P$, where $\widetilde h_{Mn}$ is the mollification considered in~\cite[Proposition 3.19]{dgz-thick-point} defined via white noise decomposition. It is easy to see that $\widetilde h_{Mn}$ differs from $h_\epsilon$ by at most $o(\log \epsilon^{-1})$ on $\mathbb{S}^\epsilon$ with probability tending to 1 as $\epsilon \to 0$. Therefore, on this event, $P$ satisfies the desired conditions. The result extends to general $\epsilon>0$ by discretizing the paths for dyadic $\epsilon$ and noting that the local fluctuation of $h_\epsilon(x)$ is at most $o(\log \epsilon^{-1})$ with probability tending to 1 as $\epsilon \to 0$.}, there exists a constant $C>0$ such that the following holds. For all $d \geq C$, with probability $1-o_\e(1)$, there exists a path $P$ in $\mathbbm{S}^\e$ connecting $\partial_L \mathbbm{S}^\e$ and $\partial_R \mathbbm{S}^\e$ such that
\begin{enumerate}[(a)]
\item The length of $P$ is at most $\e^{- 1 - A^{-1}}$;
\item For all vertices $x \in P$, we have $h_\e(x) \leq A \log \e $.
\end{enumerate}
Therefore, $D_h^{\e, \xi} (\partial_L \mathbbm{S}^\e, \partial_R \mathbbm{S}^\e; \mathbbm{S}^\e) \leq L_h^{\e, \xi} (P) \leq \e \times \e^{- 1 - A^{-1}} \times \exp( \xi A \log \e ) = \e^{\xi A - A^{-1}}$. This implies that $\lambda(d, \xi) \geq A \xi - A^{-1}$ for all $d \geq C$.\qedhere
\end{proof}

Next, we give upper and lower bounds on the derivative of $\lambda$ with respect to $\xi$, denoted by $\lambda'$. The proof will be given in Section~\ref{sec:differential}, using arguments similar to those for the two-dimensional case in~\cite{dg-fractal-dim, gp-lfpp}.

\begin{lemma}\label{lem:differential-bound}
    Fix $d \geq 3$. The function $\xi \to \lambda(\xi)$ is $\sqrt{2d}$-Lipschitz on $[0,\infty)$, and thus $\lambda'$ exists for Lebesgue-a.e.~$\xi$. Moreover, for Lebesgue-a.e.~$\xi \geq 0$, we have
    $$
    \max \big{\{} -\xi, \frac{1}{\xi} (\lambda - 1) \big{\}} \leq \lambda' \leq \sqrt{2(d-1)+2\lambda +\xi^2} - \xi.
    $$
\end{lemma}

For $d = 2$, we know that $\lambda(2, 1/\sqrt{6}) = 1/6$, so the above differential inequalities were used in~\cite{gp-lfpp} to derive sharp bounds on $\lambda$ particularly around $\xi = 1/\sqrt{6}$ as shown in Theorem~\ref{thm:bound-2d}. In contrast, for $d \geq 3$, the explicit value of $\lambda(d,\xi)$ is not known for any $\xi > 0$, so the above differential inequalities cannot yield sharper bounds than those in Theorem~\ref{thm:bound}. Nevertheless, if the value of $\lambda(d, \xi)$ were known for some $\xi$, the bound in Theorem~\ref{thm:bound} could be improved around that value.

\subsection{Monotonicity of the fractal dimension in the subcritical region}
\label{subsec:fractal-dimension}

After appropriate normalization, the metrics $D_h^{\e,\xi}$ defined in~\eqref{eq:def-metric} are expected to converge to a limiting metric as $\e \to 0$. The topology of the convergence and the behavior of the limiting metric will depend on $\xi$. In particular, it is expected that there exists a critical value $\xi_c \in (0, \infty)$ such that for $\xi \in (0, \xi_c)$ (the subcritical region), the sequence of metrics converges with respect to the uniform topology and the limiting metric induces the Euclidean topology. For $d = 2$, the convergence in the subcritical region was established in~\cite{DDDF-tightness, GM-uniqueness}, while for $d \geq 3$, only the tightness of the metrics is proved in~\cite{dgz-exponential-metric}, with the uniqueness part still open.\footnote{Note that different mollifications of the LGF are considered in~\cite{DDDF-tightness} and~\cite{dgz-exponential-metric}. Specifically, \cite{DDDF-tightness} considered the mollification of the two-dimensional GFF through convolution with the heat kernel, while \cite{dgz-exponential-metric} considered the mollification of the LGF through the white noise decomposition. However, it is expected that the convergence holds for any reasonable mollification, and moreover, the limiting metric is determined by the LGF independent of the choice of mollification; see~\cite{GM-uniqueness}.}

For the two-dimension Liouville quantum gravity (LQG) metric, there is another parameter $\gamma \in (0,2)$ which is one-to-one mapped to $\xi \in (0,\xi_c)$. More precisely, for $\gamma \in (0,2)$, let $\mathsf d_\gamma$ be the Hausdorff dimension of $\mathbb{C}$ equipped with the $\gamma$-LQG metric. Then, they satisfy the relations $\xi = \frac{\gamma}{\mathsf d_\gamma}$ and $\lambda(2, \xi) = 1 - \xi Q$, where $Q = \frac{2}{\gamma} + \frac{\gamma}{2}$; see~\cite{dg-fractal-dim, GP-KPZ-metric}. We refer to Section 2.3 of~\cite{dg-fractal-dim} for an explanation of this relation.\footnote{We can solve the equation $\lambda (2, \xi) = 1 - \xi Q$ for any $\xi \in (0,\infty)$. When $d=2$ and $\xi \in (\xi_c,\infty)$ (the supercritical region), we obtain $Q \in (0,2)$~\cite{DG-supercritical-tightness}. In this case, the sequence of metrics $D_h^{\e, \xi}$ converges with respect to the topology on lower semi-continuous functions, and the limiting metric has singular points~\cite{DG-supercritical-tightness, DG-supercritical-unique}. We refer to~\cite{bhatia2024area} for more discussions about the supercritical LQG.} The parameter $Q$ appears natural in the coordinate change formula for LQG, and for $d \geq 3$, we have $Q = \frac{d}{\gamma} + \frac{\gamma}{2}$; see~\cite[Definition 1.3]{contreras2024gaussian}. In light of these relations, we can define $\mathsf d_\gamma$ for $d \geq 3$, which is expected to be the Hausdorff dimension of the limiting metric.

\begin{defn}\label{defn:d-gamma}
    Fix $d \geq 3$. For each $\gamma \in (0,\sqrt{2d})$, define $\mathsf d_\gamma>0$ such that
    $$
    \lambda(d, \xi) = 1 - \xi Q \quad \mbox{for $\xi = \frac{\gamma}{\mathsf d_\gamma}$ and $Q = \frac{d}{\gamma} + \frac{\gamma}{2}$}.
    $$
\end{defn}

We will prove the following proposition in Section~\ref{sec:monotone}. The proof is based on Lemma~\ref{lem:differential-bound} and follows arguments similar to the two-dimensional case in~\cite[Proposition 1.7]{dg-fractal-dim}.

\begin{prop}\label{prop:dimension-increase}
    For each $\gamma \in (0,\sqrt{2d})$, there exists a unique $\mathsf d_\gamma \in (0,\infty)$ satisfying Definition~\ref{defn:d-gamma}. Moreover, as a function of $\gamma$, both $\xi$ and $\mathsf d_\gamma$ are continuous and strictly increasing on $(0,\sqrt{2d})$.
\end{prop}

The following bound for $\mathsf d_\gamma$ can be derived from Theorem~\ref{thm:bound} and Definition~\ref{defn:d-gamma}; see Figure~\ref{fig:bounds-d}.

\begin{cor} \label{cor}
    For all $d \geq 3$ and $\gamma \in (0, \sqrt{2d})$, we have $$\max \left\{ d + \gamma^2/2, \frac{6}{\sqrt{15}+4}(d + \gamma^2/2 + (\sqrt{5/2} - \sqrt{1/6})\gamma) \right\} \leq \mathsf d_\gamma \leq d +  \gamma^2/2 + \gamma \sqrt{2d-2}.$$
\end{cor}

\begin{proof}
    We first use the bound $0 \leq \lambda(\xi) \leq \xi \sqrt{2d-2}$ from Theorem~\ref{thm:bound} to obtain upper and lower bounds for $\mathsf d_\gamma$. Combining this with Definition~\ref{defn:d-gamma} yields that $0 \leq 1 - \frac{\gamma}{\mathsf d_\gamma}(\frac{d}{\gamma} + \frac{\gamma}{2}) \leq \frac{\gamma}{\mathsf d_\gamma} \sqrt{2d-2}$. Thus, it follows that $d + \gamma^2/2 \leq \mathsf d_\gamma \leq d +  \gamma^2/2 + \gamma \sqrt{2d-2}$. For $d \geq 3$ and $\gamma \in (0, \sqrt{2d})$, we have $\xi = \frac{\gamma}{\mathsf d_\gamma} \leq \frac{\gamma}{d + \gamma^2/2} \leq \frac{1}{\sqrt{2d}} \leq \frac{1}{\sqrt{6}}$. Therefore, Theorem~\ref{thm:bound} further implies that $(\sqrt{\frac{5}{2}} - \frac{1}{\sqrt{6}}) \frac{\gamma}{\mathsf d_\gamma} - \frac{\sqrt{15}-2}{6} \leq 1 - \frac{\gamma}{\mathsf d_\gamma}(\frac{d}{\gamma} + \frac{\gamma}{2})$.
    After simplification, we obtain the corollary.
\end{proof}

\begin{figure}[ht!]
\begin{center}
\includegraphics[scale=0.7]{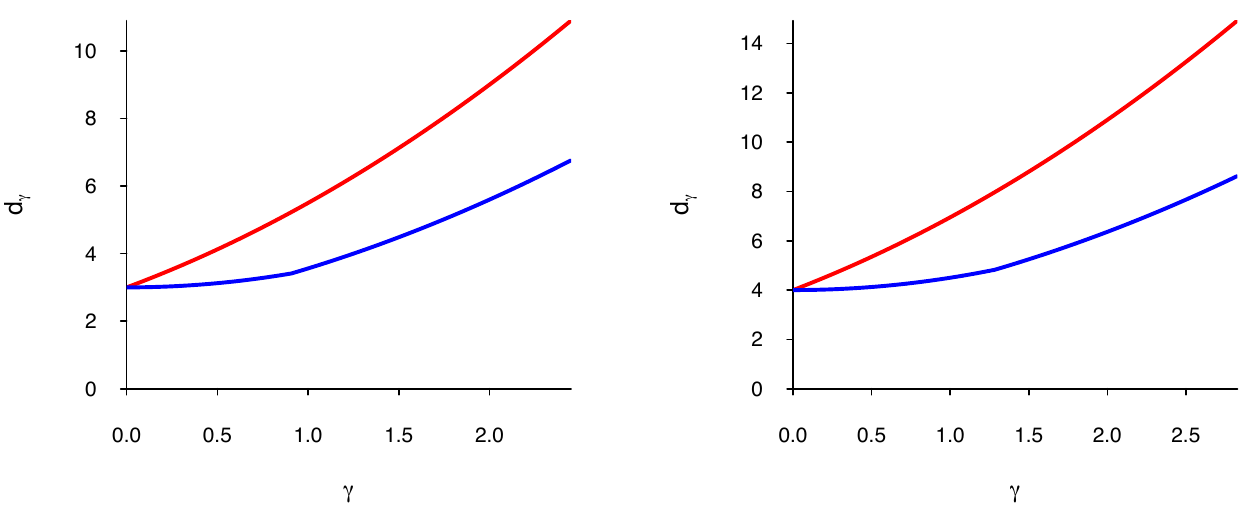}
\caption{\label{fig:bounds-d} Graph of the upper and lower bounds on $\mathsf d_\gamma$ from Corollary~\ref{cor}. \textbf{Left}: The case $d = 3$. \textbf{Right}: The case $d = 4$. 
}
\end{center}
\vspace{-3ex}
\end{figure}

\begin{remark}
We note that in \cite{cg-support}, the main theorem requires that the LQG dimension satisfies $\mathsf d_\gamma>d.$ Note that if we indeed have that $\lambda(d, \xi) = 1- \frac{\gamma}{\mathsf d_\gamma}(\frac{d}{\gamma}+\frac{\gamma}{2}),$ then this assumption follows from Corollary \ref{cor}.
\end{remark}

\begin{remark}
An interesting related question is whether the exponent $\lambda(d, \xi)$ is the same if we replace the log-correlated Gaussian field by a discrete counterpart. For instance, in the case $d = 4$, a natural discrete counterpart is the membrane model on $\mathbb{Z}^4$~\cite{Kurt-membrane}. This question is also related to the open question 4 in \cite[Section 13.1]{fgf-survey}. We plan to come back to this question in future work.
\end{remark}

\medskip
\noindent\textbf{Organization of the paper.} We will subsequently prove Lemma~\ref{lem:dimen-restriction} in Section~\ref{sec:restriction}, Lemma~\ref{lem:differential-bound} in Section~\ref{sec:differential}, and Proposition~\ref{prop:dimension-increase} in Section~\ref{sec:monotone}. 

\medskip
\noindent\textbf{Acknowledgements.} The authors thank Ewain Gwynne for suggesting the problem and for useful discussions. Z.Z.\ is partially supported by NSF grant DMS-1953848.

\section{Monotonicity of the distance exponent with respect to the dimension}\label{sec:restriction}

In this section, we prove Lemma~\ref{lem:dimen-restriction}. We first give a lemma that compares different mollifications of $h$. We assume that
\begin{equation}\label{eq:def-kernel}
    \mbox{$\rho$ is a probability measure on $\mathbb{R}^d$ such that $\iint_{\mathbb{R}^d \times \mathbb{R}^d} \log \frac{1}{|x-y|} \rho(dx) \rho(dy) < \infty$.}
\end{equation}
For $\epsilon>0$, let $h_{\epsilon, \rho} = h * \rho_\epsilon$, where $\rho_\epsilon(x) := \epsilon^{-d} \rho(\epsilon^{-1} x)$ for $x \in \mathbb{R}^d$. Then $h_{\epsilon, \rho}(x)$ can be defined simultaneously for $x \in \epsilon \mathbb{Z}^d$. It is easy to see that for $d \geq 2$, the uniform measure on a $d$-dimensional unit box $[-1,1]^d$ or on a $(d-1)$-dimensional unit box $[-1,1]^{d-1} \times \{0\}$ satisfies~\eqref{eq:def-kernel}. 

\begin{lemma}\label{lem:mollification}
    For two different measures $\rho_1, \rho_2$ satisfying~\eqref{eq:def-kernel}, let $h_{\epsilon,\rho_1}$ and $h_{\epsilon,\rho_2}$ denote the convolution of $h$ with respect to them. Then there exists a constant $C>0$ depending on $\rho_1,\rho_2$ such that with probability at least $1 - C \exp(-(\log \epsilon^{-1})^{4/3}/C)$, we have
    $$
    \sup_{x \in [0,1]^d \cap \epsilon \mathbb{Z}^d} |h_{\epsilon,\rho_1}(x) - h_{\epsilon,\rho_2}(x)| \leq (\log 
    \epsilon^{-1})^{2/3}.
    $$
\end{lemma}

\begin{proof}
    By~\eqref{eq:def-kernel}, we have $\mathbb{E}[(h_{\epsilon,\rho_1}(x) - h_{\epsilon,\rho_2}(x))^2] = O(1)$. The lemma then follows by taking a union bound.
\end{proof}

To prove Lemma~\ref{lem:dimen-restriction}, it suffices to show that $\lambda(d+1,\xi)\geq\lambda(d,\xi)$ if $d \geq 2$. Suppose that we embed $\R^d \subset \R^{d+1}$ via $(x_1, \ldots , x_d) \mapsto (x_1, \ldots , x_d , 0)$. Let $h$ be a log-correlated Gaussian field in dimension $d+1$, and let $\mathfrak h$ be a log-correlated Gaussian field in dimension $d$. Now, for $x \in \mathbb{S}^\epsilon$, define
$$
h_\epsilon(x) := \frac{1}{2^{d+1}\e^{d+1}}\int_{x + [-\epsilon, \epsilon]^{d+1}} h(z)dz.
$$
For $x \in \mathbb{S}^\epsilon \cap (\mathbb{R}^d \times \{0\})$, let $h'_\epsilon(x)$ be the field defined by
\[
h'_\epsilon(x):= \frac{1}{2^d\e^d}\int_{x + [-\epsilon, \epsilon]^d \times \{0\}} h(z)dz,
\]
and let
\[
\mathfrak h_\epsilon(x):= \frac{1}{2^d\e^d}\int_{x + [-\epsilon, \epsilon]^d \times \{0\}}\mathfrak h(z)dz.
\]
By Theorem 7.1 in~\cite{fgf-survey}, the restriction of $h$ to $\mathbb{R}^d$ has the same law as $\mathfrak h$ viewed as a function up to a global additive constant. Thus, there exists a coupling such that $h'_\epsilon(x)$ and $\mathfrak h_\epsilon(x)$ only differ by a random constant, denoted by $X$, which is independent of $\epsilon$. Furthermore, by Lemma~\ref{lem:mollification}, we know that with probability tending to 1 as $\epsilon \to 0$,
\[
  \sup_{x \in \mathbb{S}^\epsilon\cap (\R^d\times \{0\})} |h'_\epsilon(x) - h_\epsilon(x)| \leq (\log 
    \epsilon^{-1})^{2/3}.
\]
Now, let $P$ be the discrete path connecting the left and right boundaries of $\mathbb{S}^\epsilon \cap (\mathbb{R}^d \times \{0\})$ with minimal $D_{\mathfrak h}^{\e,\xi}$-length. Note that $P$ also connects the left and right boundaries of $\mathbb{S}^\epsilon$. Thus, on the above events, we have
\begin{align*}
D_h^{\e, \xi}(\partial_L \mathbbm{S}^\e, \partial_R \mathbbm{S}^\e;  \mathbbm{S}^\e) \leq D_h^{\e,\xi}(P) \leq e^{\xi ((\log 
    \epsilon^{-1})^{2/3} + X)} D_{\mathfrak h}^{\e,\xi}(P) 
\end{align*}
which together with the definition of $\lambda(d,\xi)$ proves Lemma~\ref{lem:dimen-restriction}.

\section{Bounds on \texorpdfstring{$\lambda'$}{'}}\label{sec:differential}

In this section, we prove Lemma~\ref{lem:differential-bound}. We fix $d \geq 3$ and use $\lambda(\xi)$ and $\lambda'(\xi)$ as shorthand for $\lambda(d, \xi)$ and $\lambda'(d, \xi)$, respectively. For any path $P$ in $\mathbbm{S}^\epsilon$ and $\widetilde \xi > \xi \geq 0$, we have
$$
\exp \big(-(\widetilde \xi - \xi) \sup_{x \in \mathbbm{S}^\epsilon} |h_\epsilon(x)| \big) L_h^{\epsilon, \widetilde \xi}(P) \leq L_h^{\epsilon, \xi}(P) \leq \exp \big((\widetilde \xi - \xi) \sup_{x \in \mathbbm{S}^\epsilon} |h_\epsilon(x)| \big) L_h^{\epsilon, \widetilde \xi}(P).
$$
Since $\mathbb{E}[h_\epsilon(x)^2] = -\log \epsilon +O(1)$, it follows that for any fixed $\delta>0$, $\sup_{x \in \mathbbm{S}^\epsilon} |h_\epsilon(x)| \leq - (\sqrt{2d} + \delta) \log \epsilon$ with probability at least $1-o_\epsilon(1)$. Therefore, $|\lambda(\widetilde \xi) - \lambda(\xi)| \leq \sqrt{2d}(\widetilde \xi - \xi)$, which implies that the function $\xi \to \lambda(\xi)$ is $\sqrt{2d}$-Lipschitz.

Next, we will prove that for Lebesgue-a.e. $\xi \geq 0$,
\begin{enumerate}[(i)]
    \item $-\xi \leq \lambda'(\xi) $; \label{diff-1}
    \item $\frac{1}{\xi}(\lambda(\xi) - 1) \leq \lambda'(\xi)$; \label{diff-2}
    \item $ \lambda'(\xi) \leq \sqrt{2(d-1)+2\lambda(\xi)+\xi^2} - \xi.$ \label{diff-3}
\end{enumerate}
The proof follows from arguments similar to those in the two-dimensional case. In particular, Claim~\eqref{diff-1} follows from \cite[Lemma 2.5]{dg-fractal-dim}, Claim~\eqref{diff-2} follows from \cite[Lemma 4.1]{gp-lfpp}, and Claim~\eqref{diff-3} follows from \cite[Theorem 2.1]{gp-lfpp}.

\subsection{Proof of Claim~\eqref{diff-1}}

We will prove that for any $\widetilde \xi > \xi \geq 0$, the inequality $\lambda(\widetilde \xi) \geq \lambda(\xi) + \frac{1}{2}(\xi^2 - \widetilde \xi^2)$ holds. It then follows that $0 \leq (\lambda(\xi) + \frac{1}{2} \xi^2)' = \lambda'(\xi) + \xi$, which implies Claim~\eqref{diff-1}. Fix $\widetilde \xi > \xi \geq 0$. Let $h$ and $h'$ be two independent log-correlated Gaussian fields on $\mathbb{R}^d$. Then,
$$
\widetilde h := \widetilde \xi^{-1} \Big( \xi h + \sqrt{\widetilde \xi^2 - \xi^2} h' \Big) \overset{d}{=} h.
$$
For any discrete path $P: \{1,2,\ldots, N\} \to \mathbbm{S}^\epsilon$, we have
\begin{align*}
    \mathbb{E}\Big[L_{\widetilde h}^{\epsilon, \widetilde \xi}(P) | h\Big] = \mathbb{E}\Big[ \sum_{i=1}^N \epsilon e^{\widetilde \xi \widetilde h_\epsilon(P(i))} | h \Big]&= \sum_{i=1}^N \epsilon e^{\xi h_\epsilon(P(i))} \mathbb{E} \Big[ e^{\sqrt{\widetilde \xi^2 - \xi^2} h'_\epsilon(P(i))} \Big].
\end{align*}
Since $\mathbb{E} [h'_\epsilon(x)^2] = - \log \epsilon + O(1)$, it follows that $\mathbb{E}[L_{\widetilde h}^{\epsilon, \widetilde \xi}(P) | h] \leq C \cdot L_h^{\epsilon, \xi}(P) \epsilon^{-\frac{1}{2}(\widetilde \xi^2 - \xi^2)}$ for some constant $C>0$ depending only on $\xi, \widetilde \xi$. Now, let $P$ be the path that minimizes the distance $D_h^{\epsilon, \xi}(\partial_L \mathbbm{S}^\epsilon, \partial_R \mathbbm{S}^\epsilon; \mathbbm{S}^\epsilon)$. Then, by Markov's inequality, for any fixed $\delta>0$, with probability at least $1 - o_\epsilon(1)$, we have
$$
D_{\widetilde h}^{\epsilon, \widetilde \xi}(\partial_L \mathbbm{S}^\epsilon, \partial_R \mathbbm{S}^\epsilon; \mathbbm{S}^\epsilon) \leq L_{\widetilde h}^{\epsilon, \widetilde \xi}(P) \leq D_h^{\epsilon, \xi}(\partial_L \mathbbm{S}^\epsilon, \partial_R \mathbbm{S}^\epsilon; \mathbbm{S}^\epsilon) \times \epsilon^{-\frac{1}{2}(\widetilde \xi^2 - \xi^2) - \delta}.
$$
This implies that $\lambda(\widetilde \xi) \geq \lambda(\xi) + \frac{1}{2}(\xi^2 - \widetilde \xi^2)$.

\subsection{Proof of Claim~\eqref{diff-2}}

For any $\widetilde \xi > \xi$, the function $x \to x^{\xi/\widetilde \xi}$ is concave. Thus, for any path $P$, we have $(\sum_{i=1}^N e^{\widetilde \xi h_\epsilon(P(i))})^{\xi/ \widetilde \xi} \leq \sum_{i=1}^N e^{\xi h_\epsilon(P(i))} $, which is equivalent to $(\epsilon^{-1} L_h^{\epsilon, \widetilde \xi}(P))^{\xi/ \widetilde \xi} \leq \epsilon^{-1} L_h^{\epsilon, \xi}(P)$. Combining this with~\eqref{eq:def-lambda} gives $$(\lambda(\widetilde \xi) - 1) \times \xi/ \widetilde \xi \geq \lambda(\xi)-1.$$ Therefore, $(\frac{\lambda(\xi)-1}{\xi})' \geq 0$ for Lebesgue-a.e. $\xi \geq 0$. After simplification, we obtain Claim~\eqref{diff-2}.

\subsection{Proof of Claim~\eqref{diff-3}}

We claim the following, and Claim~\eqref{diff-3} is a direct consequence of it.
\begin{prop}\label{upper}
Let $0 \leq \widetilde{\xi} \leq \xi,$ and let $\zeta>0$ be small. Then with probability tending to $1$ as $\e \to 0,$ each simple path $P$ in $[0,1]^d\cap \e \Z^d$ with $L_h^{\e,\xi}(P) \leq \e^{\lambda(\xi)-\zeta}$ satisfies
\[
L_h^{\e,\widetilde{\xi}}(P) \leq 2\e^{\lambda(\xi)-(\xi-\widetilde{\xi})(\sqrt{2(d-1)+2\lambda(\xi)+\xi^2} - \xi) - \zeta}.
\]
\end{prop}

The main lemma in proving Proposition~\ref{upper} is the following.

\begin{lemma}\label{logbound}
For $\alpha>0,$ we have
\[
\mathbb{E}(\#\{z \in [0,1]^d \cap \e\Z^d: h_\e(z)< \alpha \log \e\}) = O_\e\left(\e^{-\left(d-\frac{\alpha^2}{2}\right)}\right).
\]
\end{lemma}
\begin{proof}
The proof follows as in Lemma 3.1 in \cite{gp-lfpp} using that $\mathbb{E} h_\epsilon(z)^2 = -\log \epsilon +O(1)$.
\end{proof}
\begin{proof}[Proof of Proposition \ref{upper}]
By Lemma~\ref{logbound}, it holds with probability tending to $1$ as $\e\to 0$ that
\begin{equation}\label{small}
\# \{z \in[0,1]^d \cap \e \Z^d : h_\e(z)< \alpha \log \e\} \leq \e^{-\left(d-\frac{\alpha^2}{2}\right)-\zeta}.
\end{equation}
We will then assume that \eqref{small} holds. Let $P:\{0, \ldots , N\} \to [0,1]^d \cap \e \Z^d$ be a simple path with $L_h^{\e,\xi}(P) \leq \e^{\lambda(\xi)-\zeta}$. Then
\begin{eqnarray*}
L_h^{\e,\widetilde{\xi}}(P) &=& \sum_{j=0}^N \e e^{\widetilde{\xi}h_\e(P(j))} = \sum_{\{j:h_\e(P(j))<\alpha \log \e\}} \e e^{\widetilde{\xi}h_\e(P(j))} + \sum_{\{j:h_\e(P(j))\geq \alpha \log \e\}} \e e^{\widetilde{\xi}h_\e(P(j))}\\
&\leq & \e^{1+\alpha \widetilde{\xi}} \# \{j:h_\e(P(j))< \alpha \log \e\} + \sum_{\{j:h_\e(P(j))\geq \alpha \log \e\}} \e e^{\widetilde{\xi}h_\e(P(j))}\\
&:=& A_1+A_2.
\end{eqnarray*}
Since we are assuming \eqref{small} holds, we have that
\[
A_1 \leq \e^{\widetilde{\xi}\alpha +\frac{\alpha^2}{2} - (d-1) -\zeta}.
\]
Also, since $\widetilde{\xi} \leq \xi,$ if $h_\e(P(j))\geq \alpha \log \e,$ then
\[
e^{\widetilde{\xi}h_\e(P(j))} \leq \e^{-(\xi-\widetilde{\xi})\alpha} e^{\xi h_\e(P(j))}.
\]
Hence,
\[
L_h^{\e,\widetilde{\xi}}(P) \leq \e^{\widetilde{\xi}\alpha + \frac{\alpha^2}{2}-(d-1)-\zeta} + \e^{-(\xi-\widetilde{\xi})\alpha} L_h^{\e,\xi}(P) \leq \e^{\widetilde{\xi}\alpha + \frac{\alpha^2}{2}-(d-1)-\zeta} + \e^{\lambda(\xi)-(\xi-\widetilde{\xi})\alpha-\zeta}.
\]
Taking
\[
\alpha = -\xi + \sqrt{\xi^2 + 2\lambda(\xi)+2(d-1)}
\]
we conclude.
\end{proof}

\begin{proof}[Proof of Claim~\eqref{diff-3}]
By Proposition~\ref{upper}, for any $0 \leq \widetilde{\xi} \leq \xi$, the inequality $\lambda(\widetilde \xi) \geq \lambda(\xi)-(\xi-\widetilde{\xi})(\sqrt{2(d-1)+2\lambda(\xi)+\xi^2} - \xi)$ holds. Taking $\widetilde \xi$ to $\xi$ yields Claim~\eqref{diff-3}.
\end{proof}

\section{Well-definedness and monotonicity of the fractal dimension}\label{sec:monotone}

In this section, we prove Proposition~\ref{prop:dimension-increase}, which follows from arguments similar to the two-dimensional case in \cite{dg-fractal-dim}. For $\xi>0$, let $\widehat Q(\xi) = \frac{1}{\xi}(1 - \lambda(\xi))$. By the inequality $\lambda'(\xi) \geq \frac{1}{\xi}(\lambda(\xi)-1)$ in Lemma~\ref{lem:differential-bound}, we see that $\widehat Q(\xi)$ is a continuous and non-increasing function of $\xi$ on $(0,\infty)$. 

Define $\xi_c:= \sup \{ \xi >0 : Q(\xi) > \sqrt{2d} \}$. We first show that 
\begin{equation}\label{eq:strictly-decrease-Q}
    \mbox{the function $\xi \to \widehat Q(\xi)$ is strictly decreasing on $(0, \xi_c)$.}
\end{equation}To prove~\eqref{eq:strictly-decrease-Q}, it suffices to show that $\widehat Q'(\xi)<0$ for Lebesgue-a.e. $\xi \in (0,\xi_c)$. By Theorem~\ref{thm:bound}, we have $\lambda(\xi) \geq 0$. This is equivalent to $\xi \leq \frac{1}{\widehat Q(\xi)}$. Since $\widehat Q(\xi) > \sqrt{2d}$ for $\xi \in (0, \xi_c)$, it follows that $\widehat Q(\xi) > \xi$. Futhermore, by Lemma~\ref{lem:differential-bound}, $-\xi \leq \lambda'(\xi) = - \xi \widehat Q'(\xi) - \widehat Q(\xi)$. Combining these two inequalities, we conclude that $\widehat Q'(\xi)<0$ for Lebesgue-a.e. $\xi \in (0, \xi_c)$.

By~\eqref{eq:strictly-decrease-Q}, for each $\gamma \in (0, \sqrt{2d})$, there exists a unique $\xi \in (0,\xi_c)$ such that $\widehat Q(\xi) = \frac{d}{\gamma} + \frac{\gamma}{2}$, which implies the uniqueness of $\mathsf d_\gamma = \frac{\gamma}{\xi}$. Furthermore, it is easy to see that $\widehat Q(\xi), \xi,$ and $\mathsf d_\gamma$ are all continuous with respect to $\gamma$. Since $\widehat Q(\xi)$ is a strictly decreasing function of $\gamma$, it follows from~\eqref{eq:strictly-decrease-Q} that $\xi$ is a strictly increasing function of $\gamma$ on $(0,\sqrt{2d})$.

Finally, we show that $\mathsf d_\gamma$ is strictly increasing as a function of $\gamma$. To see this, it suffices to show that $\frac{{\rm d}\mathsf d_\gamma}{{\rm d} \xi} > 0 $ for Lebesgue-a.e. $\xi \in (0,\xi_c)$. This follows from the following calculation:
$$
    \frac{{\rm d}\mathsf d_\gamma}{{\rm d} \xi} = \frac{{\rm d} \left( \frac{1}{\xi} \Big(\widehat Q(\xi) - \sqrt{\widehat Q(\xi)^2 - 2d} \Big)\right)}{{\rm d} \xi} = -\frac{1}{\xi^2} \Big(\widehat Q(\xi) - \sqrt{\widehat Q(\xi)^2-2d}\Big)\Big(1 + \frac{\xi \widehat Q'(\xi)}{\sqrt{\widehat Q(\xi)^2-2d}}\Big).
$$
Using $\xi \widehat Q'(\xi) \leq \xi - \widehat Q(\xi)$ and $\xi \leq \frac{1}{\widehat Q(\xi)}$, we see that
$$
    1 + \frac{\xi \widehat Q'(\xi)}{\sqrt{\widehat Q(\xi)^2-2d}} \leq 1 + \frac{\frac{1}{\widehat Q(\xi)} - \widehat Q(\xi)}{\sqrt{\widehat Q(\xi)^2-2d}} < 0.
$$
Therefore, $\frac{{\rm d}\mathsf d_\gamma}{{\rm d} \xi}>0$ for Lebesgue-a.e. $\xi \in (0,\xi_c)$.

\bibliographystyle{plain}
\bibliography{ref}
\end{document}